\documentclass[graybox]{svmult}

\usepackage{mathptmx}       
\usepackage{helvet}         
\usepackage{courier}        
\usepackage{type1cm}        
\usepackage{makeidx}         
\usepackage{graphicx}        
\usepackage{multicol}        
\usepackage[bottom]{footmisc}
\usepackage{mathsym}
\usepackage{amsmath}

\makeindex             

\begin{document}

\title*{On Designs for Recursive Least Squares Residuals to Detect Alternatives}

\author{Wolfgang Bischoff}

\institute{Wolfgang Bischoff \at Katholische Universitaet Eichstaett-Ingolstadt, Mathematisch-Geographische Fakultaet, D-85071 Eichstaett, Germany, \email{wolfgang.bischoff@ku.de}}

\maketitle

\abstract{
Linear regression models are checked by a lack-of-fit (LOF)
test to be sure that the model is at least approximatively true. In many practical cases data are sampled sequentially. Such a situation appears in industrial production when goods are produced one after the other. So it is of some interest to check the regression model sequentially. This can be done by recursive least squares residuals. A sequential LOF test can be based on the recursive residual partial sum process. In this paper we state the limit of the partial sum process of a triangular array of  recursive residuals given a constant regression model when the number of observations goes to infinity. Furthermore, we state the corresponding limit process for local alternatives. For specific alternatives  designs are determined dominating other designs in respect of power of the sequential LOF test described above. In this context a result is given in which $e^{-1}$ plays a crucial role.}

\section{Introduction}
\label{sec:1}
In order to guarantee the quality of each delivery of contract goods companies take samples to decide whether the quality is or is not constant. If the goods are sequentially produced this problem can be modelled by the regression model
\begin{equation}
Y(t) = g(t) + \epsilon (t), \quad t \in [0, 1],
\label{regression model}
\end{equation}
where $g$ is the true, but unknown mean function of the quality, $ \epsilon (t)$ is a real
random variable with expectation $0$ and variance $\sigma^2>0$, and $[0,1]$ is the period of production.
Since our results keep true when $\sigma^2$ is replaced by a consistent estimator for $\sigma^2$ we can put $\sigma^2=1$ without loss of generality.

We consider the problem more generally.  We like to test if the model (\ref{regression model})  is a linear model with respect to known and linearly independent functions $f_1,\dots,f_d,$ i.e. if there exist suitable constants $\beta_1,\dots,\beta_d\in \R$ such
that $g(t)=\sum_{i=1}^d \beta_i f_i(t), t\in [0,1]$. Hence, we look for a test of the null hypothesis
\begin{equation}
H_0: g=\sum_{i=1}^d \beta_i f_i=f^\top\beta \mbox{ for some }  \beta=(\beta_1,\ldots,\beta_d)^\top\in \R^d,
\label{H_0 allg}
\end{equation}
where $f^\top=(f_1,\dots,f_d)$, against the alternative  that model (\ref{regression model}) is not a linear model with respect to $f_1,\dots,f_d,$ that is (\ref{H_0 allg}) is not fulfilled.

In quality control we are interested in testing
\begin{eqnarray}
H_0: g(t)= \beta  ={\bf 1}_{[0,1]}(t)\beta , \quad t \in [0,1],
\mbox{ $\beta\in \R$ unknown constant},
\label{H_0 data}
\end{eqnarray}
where ${\bf 1}_{[0,1]}$ is the function identical 1 on $[0,1]$, against the alternative
\begin{equation}
K: g\neq constant.
\label{K data}
\end{equation}
Such a function g under the alternative has typically the following form:
\begin{eqnarray*}
g(t)= \beta \mbox{ for } t\in [0,t_0],\;\;
g \mbox{ increasing or decreasing for } t\in (t_0,1].
\label{K specific}
\end{eqnarray*}
This form of $g$ means that the quality keeps constant up to a fixed, known or unknown change--point $t_0\in (0,1)$, then the quality is getting worse or is getting better.

In the literature on `detecting changes' in regression models, it is
common to consider (recursive) residual partial sum processes or variants thereof; see for instance,
Gardner \cite{G69}, Brown, Durbin and Evans \cite{BDE75}, Sen and Srivastava \cite{SS75}, MacNeill \cite{M78} and \cite{M78/2}, Sen \cite{S82},
Jandhyala and MacNeill \cite{MJ89}, \cite{MJ91} and \cite{MJ97}, Watson \cite{W95},
Bischoff \cite{B98}, Jandhyala, Zacks and El-Shaarawi \cite{JZES99},
Bischoff and Miller \cite{BM00}, Xie and MacNeill \cite{XM06}, Bischoff and Somayasa \cite{BS09}, Bischoff and Gegg \cite{BG11}.
The asymptotics of the partial sum of recursive residuals is investigated by Sen \cite{S82} only. Sen \cite{S82}, however, assumed a time series sampling for his asymptotic result. For our problem we need a triangular array approach.

In Section  \ref{Sect 2} we discuss some asymptotic results for the partial sum process of recursive least squares residuals. Assuming a constant regression model we state such a result for a triangular array of design points on one hand if the null hypothesis (\ref{H_0 data}) is true and on the other hand under certain assumptions if a local alternative (\ref{K data}) is true. With this we are in the position to establish an asymptotic size $\alpha$ test to test the null hypothesis  (\ref{H_0 data}). Furthermore, in Section \ref{Sect 3}, we can discuss the power of this test for certain alternatives. There, we determine designs that have uniformly more power than other designs. For one of the results $e^{-1}$ occurs as crucial constant.


\section{\bf Recursive Residuals} \label{Sect 2}

Recursive (least squares) residuals were described in Brown et al. \cite{BDE75}, for some history  see Farebrother \cite{F78}. Brown et al. considered recursive residuals for a linear regression model given a time series sampling. To this end let $t_1<t_2<\ldots$ be a sequence of (design) points (in time), let $\varepsilon_1,\varepsilon_2, \dots$ be iid real random variables with $E(\varepsilon_i)=0$ and $Var(\varepsilon_i)=1$, let $n\in \N, n>d,$ be the number of observations where $d$ is the number of known regression functions. Moreover, we put
\begin{eqnarray*}
X_{i}=(f(t_{1}),\dots,f(t_{i}))^\top, \;d\leq i\leq n,
\label{NotaionLMrec_time series}
\end{eqnarray*}
where $f=(f_1,\ldots,f_d)^\top$. So we get $n-d+1$ linear models, namely for the first $i$, $d\leq i\leq n$, observations each
$$
{\vec Y}_i= (Y_{1},\dots,Y_{i})^\top=X_{i}\vec\beta+(\epsilon_{1},\dots,\epsilon_{i})^\top,\; d\leq i\leq n.
$$
Let $t_1,\ldots,t_d$ be chosen in such a way that $rank(X_d)=d$. Then, for each $i\geq d$  we estimate $\beta$ by the least squares estimate $\hat \beta_{i}$ using the first $i$ observations ${\vec Y}_i$. Now we can define $n-d$ recursive residuals
$$
e_{i}=\frac{Y_{i}-f(t_{i})^\top\hat\beta_{i-1}}{\left(1+f(t_{i})^\top (X_{i-1}^\top X_{i-1})^{-1}f(t_{i})\right)^{1/2}}, \;i=d+1,\ldots,n.
$$
To state Sen's and our result it is convenient to define the partial sum operator
$T_n :\R^n \longrightarrow C[0,1], {\bf a} = (a_1, \ldots, a_n)^\top \mapsto T_n
({\bf a})(z),~~z\in [0,1],$ where
$$
T_n ({\bf a}) (z) = \sum^{[nz]}_{i=1} a_i + (nz- [nz]) a_{[nz]+1},~z\in [0,1].
$$
Here we used $[s] = \max \{ n \in\N_0 \mid n\leq s\}$ and $\sum^0_{i=1} a_i=0$.
Let ${\bf a}=(a_1,\dots,a_n)^\top\in \R^n, b_i=a_1+\ldots+a_i, i=1,\ldots,n,$ then the function $T_n({\bf a})(\cdot)$ is shown in Figure 2.
\unitlength8.5mm
\begin{center}
\begin{picture}(10,3.5)
\put(0,0){\line(1,0){10}}
\put(0,-0.7){$0$}

\put(0,-0.2){\line(0,1){0.4}}
\put(1,-0.15){\line(0,1){0.3}}

\put(.8,-0.7){$1/n$}

\put(2,-0.15){\line(0,1){0.3}}

\put(1.8,-0.7){$2/n$}

\put(3,-0.15){\line(0,1){0.3}}

\put(2.8,-0.7){$3/n$}

\put(4,-0.15){\line(0,1){0.3}}
\put(5,-0.15){\line(0,1){0.3}}
\put(6,-0.15){\line(0,1){0.3}}
\put(7,-0.15){\line(0,1){0.3}}
\put(8,-0.15){\line(0,1){0.3}}
\put(9,-0.15){\line(0,1){0.3}}
\put(10,-0.2){\line(0,1){0.4}}

\put(9.9,-0.7){$1$}

\put(.01,.02){\line(1,1){1}}
\put(.65,1.04){\line(1$b_1$,2){1}}
\put(1.65,3.06){\line(1$b_2$,-1){1}}
\put(3.2,2.1){\line(0$b_3$,0){0}}
\put(3.8,1.2){\line(0$.~~ . ~~. ~~. ~~. ~~.$,0){0}}

\put(8.2,1.6){\line(0$b_{n-1}$,0){0}}  \put(8.95,1.8){\line(1,-1){1}}
\put(9.5,.6){\line(0$b_n$,0){0}}

\put(3,-1.5){{\footnotesize{\bf Fig 2.} The function $T_n({\bf a})(\cdot)$.}}
\end{picture}
\end{center}

\setcounter{figure}{2}

\vspace*{13mm}

\noindent  By Donsker's Theorem the stochastic process $\frac{1}{\sqrt{n}}T_n((\mbox{\boldmath$\varepsilon_1,\ldots,\epsilon$}_n)^\top)$  converges weakly
to Brownian motion $B$ for $n \to \infty $.
For recursive residuals Sen \cite{S82} proved the following result.

\begin{theorem} [Sen \cite{S82}]
For the regression model given in (\ref{H_0 allg}) let $\vec{e}_n=(e_{d+1},\ldots,e_n)^\top$. If $H_0$ given in (\ref{H_0 allg}) is true, then under certain assumptions it holds true
$$
\frac{1}{\sqrt{n-d}}T_{n-d}(\vec{e}_n)(z) \mbox{ converges weakly to } B(z),\; z \in[0,1], \mbox{ for } n\to \infty.
$$
\end{theorem}
Sen, however, could not determine the limit process for a local alternative.

The time series sampling approach described above cannot be applied to problems of experimental design. Instead we need the asymptotic result for a triangular array of design points under the null hypothesis and under local alternatives.
To this end let $n_0 \in \N$, $n_0>d$, be the number of observations. We assume that the data are taken at the design points $0\leq t_{n_01}\leq t_{n_02}\leq \ldots \leq t_{n_0n_0}\leq 1.$ These design points can be embedded in a triangular array of design points: $0\leq t_{n1}\leq t_{n2}\leq \ldots \leq t_{nn}\leq 1,\;n \in\N.$ Furthermore, let $\varepsilon_{n1},\dots,\varepsilon_{nn}, n\in \N,$ be a triangular array of real random variables where $\varepsilon_{n1},\dots,\varepsilon_{nn}$ are iid with  $\E(\varepsilon_{ni})=0$ and $Var(\varepsilon_{ni})=1$ for each $n\in \N$.
Accordingly, we get a corresponding triangular array of observations for model (\ref{regression model}):
$$
Y_{nj} =g (t_{nj}) + \epsilon_{nj}, ~1 \leq j \leq n,~n \in \N.
$$
We put
\begin{eqnarray*}
\varepsilon_i^{n}=(\varepsilon_{n1},\dots,\varepsilon_{ni})^\top, \;d\leq i\leq n,\\
X^n_{i}=(f(t_{n1}),\dots,f(t_{ni}))^\top, \;d\leq i\leq n.
\end{eqnarray*}
So we get $n-d+1$ linear models under the null hypothesis $H_0$ given in (\ref{H_0 allg}), namely for the first $i$, $d\leq i\leq n$, observations each
$$
{\bf Y}^n_i= (Y_{n1},\dots,Y_{ni})^\top=X^n_{i}\beta+\varepsilon_i^{n}.
$$
Let $t_{n1},\ldots,t_{nd}$ for all $n\geq d$ be chosen in such a way that $rank(X^n_{d})=d$. Moreover, let  $\hat \beta_{i}^n, d\leq i,$ be the least squares estimate for $\beta$ using the first $i$ observations ${\bf Y}^n_i$. Then the $n-d$ recursive least squares residuals for the triangular array are defined by
$$
e_{ni}=\frac{Y_{ni}-f(t_{ni})^\top\hat\beta_{i-1}^n}{\left(1+f(t_{ni})^\top (X_{ni-1}^\top X_{ni-1})^{-1}f(t_{ni})\right)^{1/2}},\;i=d+1,\dots,n.
$$
Assuming the constant regression model the next result states the limit of the recursive residual partial sum process if $H_0$, see (\ref{H_0 data}), is true and if a local alternative is true. In case (\ref{H_0 data}) is true, the location of the design points has no influence. So the result is true for any  triangular array of design points. If a local alternative is true, we give the result for a uniform array of design points only to avoid further technical notation. This result will be sufficient for our purposes below.
\begin{theorem} [Master Thesis Rabovski \cite{R03} under the supervision of the author and Frank Miller]
For the constant regression model let $\vec{e}^n=(e_{nd+1},\ldots,e_{nn})^\top$ be the vector of the $n-d$ recursive residuals of a triangular array of design points.
\begin{itemize}
\item[a)] If $H_0$ given in  (\ref{H_0 data}) is true, then for any triangular array of design points
$$
\frac{1}{\sqrt{n-d}}T_{n-d}(\vec{e}^n)(z) \mbox{ converges weakly to } B(z),\; z \in[0,1], \mbox{ for } n\to \infty.
$$
\item[b)] Let $g:[0,1]\to\R, g\neq constant$, have bounded variation and let the triangular array of design points be given by a uniform design
$$
t_{n1}=0,t_{n2}=\frac{1}{n-1},t_{n3}=\frac{2}{n-1},\ldots,t_{nn}=1, n\in \N.
$$
Then, if the local alternative $\frac{1}{\sqrt {n-d}}g$ is true,
$$
\frac{1}{\sqrt{n-d}}T_{n-d}(\vec{e}^n)(z) \mbox{ converges weakly to }  h(z)  +B(z)
 ,\; z \in[0,1], \mbox{ for } n\to \infty,
$$
where
$$
h(z)=\int_0^z g(t) dt-\int_0^z \frac{1}{s} \int_0^s g(t) dt ds, \;z\in [0,1].
$$
\end{itemize}
\label{asympt constant regr}
\end{theorem}
Theorem \ref{asympt constant regr} part a) can be used to establish an asymptotic size $\alpha$ test of Kolmogorov(-Smirnov) or Cram\'{e}r-von Mises type. As an example we state a one-sided test of Kolmogorov type, to detect a negative deviation $h$ from Brownian motion.
\begin{theorem}
For the constant regression model let $\vec{e}^n=(e_{nd+1},\ldots,$ $e_{nn})^\top$ be the vector of the $n-d$ recursive residuals of an arbitrary triangular array of design points. Then an asymptotic size $\alpha$ test is given by
$$
\mbox{reject } H_0 \mbox{ given in (\ref{H_0 data}) }\;\Longleftrightarrow \exists t \in[0,1]: \frac{1}{\sqrt{n-d}}T_{n-d}(\vec{e}^n)(t)<\Phi^{-1}(\frac{\alpha}{2}),
$$
where $\Phi^{-1}(\frac{\alpha}{2})$ is the $\alpha/2$ quantile of the standard normal distribution.
\label{Kolmogorov test constant regr}
\end{theorem}
\begin{proof}
Since $P(\exists t \in[0,1]: B(t)<\Phi^{-1}(\frac{\alpha}{2}))=\alpha$, see, for instance, Shorack \cite{S00} p.314, the theorem is proved.
\end{proof}
The above test is not constructed sequentially. The test statistic $ \frac{1}{\sqrt{n-d}}T_{n-d}(\vec{e}^n)(t)$ can be calculated sequentially for each new recursive residual $e_{ni}$ and so the null hypothesis can be rejected as soon as the test statistic is less than $\Phi^{-1}(\frac{\alpha}{2})$.

\section{Designs for Detecting Alternatives}\label{Sect 3}

We look for designs being useful for the quality problem discussed in the introduction. Therefore we consider the constant regression model. Usually in the context of quality control certain properties of the alternative are often known.

We begin with the alternative
\begin{equation}
g_{t_0}(t)=g_{t_0;c_0,c_1}(t)=c_0{\bf 1}_{[0,t_0]}(t)+c_1{\bf 1}_{(t_0,1]}(t),\;t\in [0,1],
\label{altern jump}
\end{equation}
where $c_0,c_1\in\R$ are unknown constants and the change-point $t_0\in (0,1)$ is a known or unknown constant. We assume $c_0>c_1$ to get a negative trend $h$, see Theorem \ref{theorem 4}. (Note that Theorem \ref{Kolmogorov test constant regr} states a test for detecting negative trends $h$.) Let a triangular array of design points be given with
$$
q:=\lim_{n\to\infty}\frac{\mbox{ number of } \{t_{ni}|t_{ni}\leq t_0, 1\leq i\leq n\}}{n}\in(0,1).
$$
We call such a triangular array of design points an asymptotic $q$-design. The proof of the following result is given in the next section.

\begin{theorem}\label{theorem 4}
For a constant regression model let $\vec{e}^n=(e_{nd+1},\ldots,e_{nn})^\top$ be the vector of the $n-d$ recursive residuals of a triangular array of design points being an asymptotic $q$-design. Let the alternative $g_{t_0}$ given in (\ref{altern jump}) be true. Then, we have for the local alternative $\frac{1}{\sqrt {n-d}}g_{t_0}$:
$$
\frac{1}{\sqrt{n-d}}T_{n-d}(\vec{e}^n)(z) \mbox{ converges weakly to }  h(z)  +B(z)
 ,\; z \in[0,1], \mbox{ for } n\to \infty,
$$
where
$$
h(z)=q(c_1-c_0)(\ln(z)-\ln(q)){\bf 1}_{(q,1]}(z), \;z\in [0,1].
$$
\label{Kolmogorov test constant regr altern}
\end{theorem}
For an asymptotic $q$-design the power of the test given in Theorem \ref{Kolmogorov test constant regr} with respect to the alternative (\ref{altern jump}) is given by
\begin{eqnarray*}
P(\exists z\in[0,1]:B(z)-q(c_0-c_1)(\ln(z)-\ln(q)){\bf 1}_{(q,1]}(z)\leq \Phi^{-1}(\frac{\alpha}{2})).
\nonumber
\end{eqnarray*}
Therefore, we call an asymptotic $q^*$-design uniformly better than an asymptotic $q$-design if for all $z\in(0,1]$
\begin{eqnarray}\label{power const-const comparision}
-q^*(\ln(z)-\ln(q^*)){\bf 1}_{(q^*,1]}(z) \leq
-q(\ln(z)-\ln(q)){\bf 1}_{(q,1]}(z) \\
\mbox{with } '<' \mbox{ at least for one } z \in [0,1]. \nonumber
\end{eqnarray}
The proof of the following result is given in the next section.
\begin{theorem}\label{theorem 5}
Let the situation considered in Theorem 4 be given and let $q_1, q_2\in [e^{-1},1)$. Then an asymptotic $q_1$-design is uniformly better than an asymptotic $q_2$-design, if $q_1<q_2$.
\end{theorem}
The author does not know whether the result stated above has some relation to the famous $e^{-1}$-law for the best choice problem, see, for instance, Bruss \cite{B84}. For $n_0$ design points we consider the design $d^*$ with the fractional part of about $e^{-1}$ design points as near as possible at $0$ (let $t^*_1$ be the largest of these design points) and the fractional part of about $1-e^{-1}$ design points as near as possible at $1$ (let $t^*_2$ be the smallest of these design points). Then, by Theorem \ref{theorem 5}, $d^*$ is asymptotically the uniformly best applicable design, if $t^*_1<t^*_2$ and $t^*_1\leq t_0$.

Finally, we consider the alternative
\begin{equation}
g(t)=c_0{\bf 1}_{[0,t_0]}(t)+(c_0+c_1t_0-c_1t){\bf 1}_{(t_0,1]}(t),\;t\in [0,1],
\label{altern straight-line}
\end{equation}
where $t_0\in[0,1)$ is a known or unknown constant and $c_0\in\R,c_1\in(0,\infty)$ are unknown constants. The last result follows in an analogous way as above.
\begin{theorem}
The asymptotic $0$-design which is uniformly distributed on $[t_0,1]$ is uniformly better with respect to the alternative (\ref{altern straight-line}) than an asymptotic $q$-design, $q\in (0,1)$, whose fractional part of design points on $[t_0,1]$ is uniformly distributed.
\end{theorem}
For an unknown change-point $t_0$ the above result is of theoretical interest only.
\section{Some Proofs}\label{Sect 4}

The following relation between an arbitrary design and a uniform design is crucial for the next proof. To this end let the alternative (\ref{altern jump}) and an arbitrary triangular array of design points $t_{n1}, \ldots, t_{nn}$ with $0 \leq t_{n1} \leq \ldots \leq t_{ns} \leq t_{0} < t_{ns+1} \leq \ldots \leq t_{nn} \leq 1$ be given. Moreover, let $q := s/n$. Then we have
$$
g_{t_0}(t_{ni}) = g_q \left( \frac{i-1}{n-1} \right), \; i = 1,\ldots,n.
$$
Thus instead to analyze the alternative $g_{t_0}$ and an arbitrary design with $s$ design points equal to or less than $t_0$ we can analyze the alternative $g_q$ with the change-point $q=s/n$ and a uniform design.


\begin{proof}[of Theorem \ref{theorem 4}]
By the above considerations the limit process of the recursive residual partial sum process with respect to the local alternative $\frac{1}{\sqrt{n-d}}g_q$ and a uniform design coincides with the limit process with respect to the local alternative $\frac{1}{\sqrt{n-d}}g_{t_0}$ and an asymptotic $q$-design. The trend $h$ given in Theorem \ref{asympt constant regr} part b) can be obtained for the local alternative $\frac{1}{\sqrt{n-d}}g_q$ and a uniform design after some calculations:
\begin{eqnarray*}
&h(z)=0, &\;z\in[0,q], \\
&h(z)=\int_{q}^z c_1 - \frac{1}{s} (qc_0+(s-q)c_1) ds
=q(c_1-c_0)(\ln(z)-\ln(q)), &\;z\in (q,1].
\end{eqnarray*}
\end{proof}

\begin{proof}[of Theorem \ref{theorem 5}]
For $z=1$ the expression $-q(\ln(z)-\ln(q)){\bf 1}_{(q,1]}(z)$ considered in (\ref{power const-const comparision}) takes on its minimum for $q=e^{-1}$ and, furthermore, it is strictly increasing on $[e^{-1},1)$.

Let $e^{-1}\leq q_1<q_2<1$. Then we have for all $z\in (q_2,1]$
$$
\frac{d}{dz}(-q_1(\ln(z)-\ln(q_1)))=-\frac{q_1}{z}>-\frac{q_2}{z}=\frac{d}{dz}(-q(\ln(z)-\ln(q))).
$$
This together with the first result of the proof provides the statement of Theorem  \ref{theorem 5}.
\end{proof}

\end{document}